\documentclass[10pt]{amsart}
\usepackage[latin1]{inputenc}
\usepackage{amsmath}
\usepackage{amsfonts}
\usepackage{amssymb}
\usepackage{amsthm}
\usepackage{graphicx,epsfig}
\usepackage{amscd}
\DeclareMathOperator{\re}{Re}
\DeclareMathOperator{\im}{Im}

\DeclareMathOperator{\isomzero}{Isom_0}

\newcommand{\nil}{\mathrm{Nil}_3}

\newcommand{\su}{\mathrm{SU}_{1,1}}
\newcommand{\psu}{\mathrm{PSU}_{1,1}}
\newcommand{\psl}{\widetilde{\mathrm{PSL}}_2(\R)}

\newcommand{\D}{\mathbb{D}}
\newcommand{\M}{\mathbb{M}}
\newcommand{\R}{\mathbb{R}}
\newcommand{\h}{\mathbb{H}}

\newcommand{\C}{\mathbb{C}}

\newcommand{\LL}{\mathbb{L}}
\newcommand{\E}{\mathbb{E}}

\newcommand{\rmT}{\mathrm{T}}

\newcommand{\rmd}{\mathrm{d}}

\newcommand{\rmU}{\mathrm{U}}

\newcommand{\rmI}{\mathrm{I}}

\newcommand{\cS}{{\mathcal S}}

\newcommand{\cE}{{\mathcal E}}

\newcommand{\cA}{{\mathcal A}}

\def\flecha{\rightarrow}
\let\8=\infty
\def\S{\mathbb{S}}

\numberwithin{equation}{section}

\begin{document}

\newtheorem{thm}{Theorem}[section]
\newtheorem*{thmintro}{Theorem}
\newtheorem{cor}[thm]{Corollary}
\newtheorem{prop}[thm]{Proposition}
\newtheorem{app}[thm]{Application}
\newtheorem{lemma}[thm]{Lemma}
\newtheorem{notation}[thm]{Notations}
\newtheorem{hypothesis}[thm]{Hypothesis}

\newtheorem{defin}[thm]{Definition}
\newenvironment{defn}{\begin{defin} \rm}{\end{defin}}
\newtheorem{remk}[thm]{Remark}
\newenvironment{rem}{\begin{remk} \rm}{\end{remk}}
\newtheorem{exa}[thm]{Example}
\newenvironment{ex}{\begin{exa} \rm}{\end{exa}}
\newtheorem{cla}[thm]{Claim}
\newenvironment{claim}{\begin{cla} \rm}{\end{cla}}

\title{The Gauss map of surfaces in $\psl$}

\author{Beno\^\i t Daniel}
\address{Universit\'e de Lorraine\\
Institut \'Elie Cartan de Lorraine\\
UMR 7502\\
CNRS\\
B.P. 70239\\
F-54506 Vand\oe{}uvre-l\`es-Nancy cedex\\
FRANCE\\
and Korea Institute for Advanced Study\\
School of Mathematics\\
85 Hoegiro\\
Dongdaemun-gu\\
Seoul 130-722\\
REPUBLIC OF KOREA}
\email{benoit.daniel@univ-lorraine.fr}

\author{Isabel Fern\'andez}
\address{Departamento de Matem\'atica Aplicada I\\
Universidad de Sevilla\\
E-41012 Sevilla\\
SPAIN}
\email{isafer@us.es}

\author{Pablo Mira}
\address{Departamento de Matem\'atica Aplicada y Estad\'\i stica\\
Universidad Polit\'ecnica de Cartagena\\
E-30203 Cartagena\\
Murcia\\
SPAIN}
\email{pablo.mira@upct.es}

\date{}



\begin{abstract}
We define a Gauss map for surfaces in the universal cover of the Lie group $\mathrm{PSL}_2(\R)$ endowed with a left-invariant Riemannian metric having a $4$-dimensional isometry group.  This Gauss map is not related to the Lie group structure. We  prove that the Gauss map of a nowhere vertical surface of critical constant mean curvature is harmonic into the hyperbolic plane $\h^2$ and we obtain a Weierstrass-type representation formula. This extends results in $\h^2\times\R$ and the Heisenberg group $\nil$, and completes the proof of existence of  harmonic Gauss maps for  surfaces of critical constant mean curvature in any homogeneous manifold diffeomorphic to $\R^3$ with isometry group of dimension at least $4$.
\end{abstract}

\maketitle

\section{Introduction}

An important property in the theory of constant mean curvature (CMC)
surfaces is that the Gauss map (i.e., the unit normal vector viewed
as a map into the $2$-sphere $\S^2$) of a CMC surface in $\R^3$ is a harmonic map
into $\S^2$. If the surface in $\R^3$ is minimal, the Gauss map is
holomorphic.

When we consider an immersed oriented surface $\Sigma$ in a
Riemannian three-manifold $M$, its unit normal vector is a map
$N:\Sigma\to\rmU M$ that takes values in the five-dimensional unit
tangent bundle $\rmU M$ of $M$. It is then unclear how to define for
$\Sigma$ a Gauss map into $\S^2$ that provides relevant information
about the surface. The construction of such a Gauss map will depend
on the geometric properties of $M$, and in general the Gauss map will not be
harmonic for CMC surfaces in $M$. Nonetheless, for some ambient
spaces $M$ such a Gauss map has been
constructed, and has provided key information about the global
properties of certain classes of surfaces in $M$.

An illustration of this fact is the case of CMC surfaces in hyperbolic $3$-space $\h^3$. There is a well known
way to define a ``hyperbolic Gauss map'' for surfaces in $\h^3$
\cite{epstein,bryant} taking values in the asymptotic boundary
$\partial_{\infty}\h^3\simeq\bar\C$ where $\bar\C=\C\cup\{\infty\}$
is the Riemann sphere. R. Bryant \cite{bryant} proved that this
hyperbolic Gauss map is holomorphic for CMC $1$ surfaces in $\h^3$
and he obtained a Weierstrass-type representation formula for these
surfaces; these techniques were then developed by M. Umehara and K.
Yamada \cite{uy} and many other authors, leading to very important
improvements in the study of these surfaces. See also \cite{galvezmathann,egm} for
other applications of the hyperbolic Gauss map.

Recently, much attention was drawn to surfaces in simply connected
homogeneous $3$-manifolds. The homogeneous spaces diffeomorphic to
$\R^3$ with a $4$-dimensional isometry group constitute a
two-parameter family; they will be denoted by $\mathbb{E}^3(\kappa,\tau)$ where $\kappa \leqslant 0$ and 
$(\kappa,\tau)\neq(0,0)$. When $\tau=0$ we get the product space
$\h^2 (\kappa)\times \R$; when $\kappa=0$ we obtain the $3$-dimensional Heisenberg
group ${\rm Nil}_3$ endowed with a left-invariant metric; otherwise we obtain the universal cover of
the Lie group ${\rm PSL}_2 (\R)$ endowed with a left-invariant
metric; we will denote it by $\psl$. Thus, we see that the metrics
of $\h^2\times\R$ and $\nil$ arise as limits of certain left
invariant metrics on $\psl$.

The second and third authors \cite{fmajm} defined a ``hyperbolic
Gauss map'' taking values in the hyperbolic plane $\h^2$ for
surfaces in $\h^2\times\R$ having a regular projection on $\h^2$;
they moreover proved that for CMC $1/2$ surfaces this Gauss map is
harmonic, and they obtained a Weierstrass-type representation
formula for these surfaces. Then, in  $\nil$, the first author
\cite{danielimrn} studied a Gauss map defined using the Lie group
structure and obtained obtained a Weierstrass-type representation
formula for minimal surfaces in $\nil$. These two Gauss maps have
quite different definitions and properties. Moreover, the use of
these Gauss maps has led to important discoveries in their
respective theories, such as solutions to Bernstein type problems,
classification of complete multigraphs or half-space theorems \cite{fmtams,dh}.

Let us indicate that minimal surfaces in $\R^3$ and $\nil$, CMC $1$
surfaces in $\h^3$ and CMC $1/2$ surfaces in $\h^2\times \R$ all
have \emph{critical constant mean curvature}. That is, their mean
curvature equals the value $c$ such that there exist compact CMC $H$
surfaces respectively in $\R^3$, $\nil$, $\h^3$, $\h^2\times \R$ if and only if
$|H|>c$. By the same definition, the value $c$ for the critical mean
curvature in $\psl =\mathbb{E}^3 (\kappa,\tau)$ for $\kappa<0$ and $\tau
\neq 0$, is $c=\sqrt{-\kappa}/2$. Moreover, CMC surfaces with
critical mean curvature in $\h^2\times\R$, $\nil$ and $\psl$ are
related by a Lawson-type correspondence \cite{danielcmh}.

The aims of this paper are:
\begin{itemize}
 \item to define geometrically a smooth Gauss map for smooth surfaces in $\psl$ taking values in a $2$-sphere, in a way that does not depend on a choice of model and such that the Gauss map is ``compatible'' with ambient isometries (in a sense that will be made precise),
\item to try to unify the definitions of the Gauss map of surfaces in $\h^2\times\R$, $\psl$ and $\nil$,
\item to prove that the Gauss maps of CMC local graphs with critical mean curvature in $\psl$ are harmonic maps into $\h^2$ (local graphs will be defined in Section \ref{sec:ekappatau}).
\end{itemize}

Let us first eliminate two approaches that could seem natural to solve this problem.


The first approach would be to define a left-invariant Gauss map using the Lie group structure, as it is implicitely done in $\R^3$ and as it was done in $\nil$ \cite{danielimrn} (see also \cite{dm}). Isometries coming from the Lie group structure in $\h^2\times\R$ and $\psl$ do not have a geometric characterization (see Remark \ref{rem:liegroup}) as in $\R^3$ and $\nil$ (where they are ``translations''). Moreover, the homogeneous manifold $\psl$ is isometric to another Lie group with a left-invariant metric (see Theorem 2.14 and Corollary 3.19 in \cite{mpliegroups} for a discussion of these Lie group structures), so defining a left-invariant Gauss map would depend on the choice of a Lie group structure. These left-invariant Gauss maps do not relate well to isometries. One can compute that, for surfaces with critical CMC, they satisfy some second order elliptic partial differential equations that are not harmonic map equations.


The second approach would be to define a Gauss map taking values in
the asymptotic boundary of the ambient manifold, as it is again
implicitly done in $\R^3$ and as it was done in $\h^3$. However,
even in $\h^2\times\R$ this is not suitable: indeed, $\h^2\times\R$
admits no differentiable Hadamard compactification \cite{kloeckner},
and hence the Gauss map of a smooth surface would not necessarily be
smooth.

The paper is organised as follows. In Section \ref{sec:ekappatau}, we introduce some preliminary material about the manifolds $\E^3(\kappa,\tau)$. In Section \ref{sec:gauss}, we prove the existence and the uniqueness of a Gauss map satisfying some geometric conditions; this Gauss map reaches the first two of the abovementioned aims of the paper. In Section \ref{sec:conformal}, we gather some preliminary computations related to conformal immersions and their Gauss map. Section \ref{sec:critical} is devoted to the Gauss map of surfaces with critical CMC: we prove that it is harmonic into $\h^2$ for local graphs and we obtain a Weierstrass-type representation theorem; hence this Gauss map reaches the third of the abovementioned aims. We also relate the Gauss map of the surface with the Gauss map of its sister minimal surface in $\nil$. 
In Section \ref{sec:lorentzian} we give the expression of the Gauss map  using the Lorentzian model for $\h^2$.

\section{The manifolds $\E^3(\kappa,\tau)$} \label{sec:ekappatau}

The simply connected homogeneous Riemannian $3$-manifolds with a
$4$-dimensional isometry group constitute, together with Euclidean
$3$-space and round $3$-spheres, a $2$-parameter family
$(\E^3(\kappa,\tau))_{(\kappa,\tau)\in\R^2}$, such that there exists
a Riemannian fibration $\pi:\E^3(\kappa,\tau)\to\M^2(\kappa)$ with
bundle curvature $\tau$, where $\M^2(\kappa)$ is the simply
connected surface of curvature $\kappa$. The fibration is a product
fibration if and only if $\tau=0$. We refer to \cite{danielcmh} and
references therein for more details.

There are different types of manifolds according to the values of $\kappa$ and $\tau$. In this paper we will only consider the case where
$$\kappa\leqslant0.$$ The manifold $\E^3(\kappa,\tau)$ is diffeomorphic to $\R^3$ and isometric
\begin{itemize}
 \item to Euclidean space $\R^3$ when $(\kappa,\tau)=(0,0)$; in this case the isometry group has dimension six, and the fibration $\pi$ is not unique,
\item to the $3$-dimensional Heisenberg group endowed with a left-invariant metric when $\kappa=0$ and $\tau\neq0$; we will denote it by $\nil(\tau)$ (up to dilations, all these metrics are isometric),
\item to $\h^2(\kappa)\times\R$ when $\kappa<0$ and $\tau=0$, where $\h^2(\kappa)$ is the hyperbolic plane of constant curvature $\kappa$,
\item to the universal cover of $\mathrm{PSL}_2(\R)$ endowed with certain left-invariant metrics when $\kappa<0$ and $\tau\neq0$; we will denote it by $\psl$ (up to dilations, all these metrics are isometric to one such that $\kappa-4\tau^2=-1$).
\end{itemize}

A vector $v$ is said to be \emph{vertical} if $\rmd\pi(v)=0$ and \emph{horizontal} if it is orthogonal to vertical vectors. We let $\xi$ be a unit vertical field in $\E^3(\kappa,\tau)$ (the field $\xi$ is unique up to multiplication by $-1$). We say that a vector $v$ is \emph{upwards pointing} (respectively, \emph{downwards pointing}) if $\langle v,\xi\rangle>0$ (respectively, $\langle v,\xi\rangle<0$).

We set $$c:=\frac{\sqrt{-\kappa}}2.$$ This constant is called the \emph{critical mean curvature} because there exist compact CMC $H$ surfaces in $\E$ if and only if $|H|>c$. When $\kappa<0$, \emph{horocylinders}, i.e. inverse images by $\pi$ of  horocycles of $\h^2(\kappa)$, have constant mean curvature $c$. When $\kappa=0$, \emph{vertical planes}, i.e. inverse images by $\pi$ of straight lines of $\R^2$, are minimal, and hence have constant mean curvature $c$. In particular, any $\E^3(\kappa,\tau)$ with $\kappa\leqslant0$ can be foliated by topological planes of critical CMC, all of them congruent to each other.

A surface is said to be \emph{nowhere vertical} (or a \emph{local graph}) if $\xi$ is nowhere tangent to it,
that is, if the restriction of $\pi$ to the surface is a local diffeomorphism (in $\h^2(\kappa)\times\R$ this means that the surface has a regular projection on $\h^2(\kappa)$); it is said to be an \emph{entire graph} if the restriction of $\pi$ to the surface is a global diffeomorphism onto $\M^2(\kappa)$.
The \emph{angle function} of an oriented surface $\Sigma$ is the function $\langle N,\xi\rangle$ where $N$ is the unit normal vector field to $\Sigma$. Hence, a surface is nowhere vertical if and only if its angle function does not vanish. CMC local graphs, and in particular entire graphs, play an important role among CMC surfaces; see for instance \cite{fmtams,hrs0,hrs2,dh,mrr,manzanopr,mazetpsl}.

For $\rho>0$ we set $$\D(\rho)=\{z\in\C;|z|<\rho\}.$$ We also set $$\D=\D(1).$$ We use the following model for the hyperbolic plane of constant curvature $\kappa$: $$\h^2(\kappa)=\D\left(\frac2{\sqrt{-\kappa}}\right)$$ endowed with the metric given in canonical coordinates $(x_1,x_2)$ by
$$\Lambda^2(\rmd x_1^2+\rmd x_2^2)$$
where $$\Lambda=\frac1{1+\frac\kappa4(x_1^2+x_2^2)}=\frac1{1-c^2|\zeta|^2}$$
with $$\zeta=x_1+ix_2.$$

The model we use for $\E^3(\kappa,\tau)$ is
$\D\left(\frac2{\sqrt{-\kappa}}\right)\times\R$ when $\kappa<0$ or $\R^3$ when $\kappa=0$, endowed with the metric given in canonical coordinates $(x_1,x_2,x_3)$ by
$$\Lambda^2(\rmd x_1^2+\rmd x_2^2)+(\tau\Lambda(x_2\rmd x_1-x_1\rmd x_2)+\rmd x_3)^2.$$
In this model we have $$\pi(x_1,x_2,x_3)=(x_1,x_2).$$

We consider the following orthonormal frame:
\begin{equation} \label{defframeV}
 V_1=\frac1\Lambda\frac\partial{\partial x_1}-\tau x_2\frac\partial{\partial x_3},\quad
V_2=\frac1\Lambda\frac\partial{\partial x_2}+\tau x_1\frac\partial{\partial x_3},\quad
V_3=\frac\partial{\partial x_3}=\xi.
\end{equation}
The fields $V_1$ and $V_2$ are the horizontal lifts in $\E^3(\kappa,\tau)$ by the fibration $\pi$ of the fields $\frac1\Lambda\frac\partial{\partial x_1}$ and $\frac1\Lambda\frac\partial{\partial x_2}$ in $\M^2(\kappa)$.


\section{Definition of the Gauss map} \label{sec:gauss}

In this section we set $$\kappa<0.$$ We will define a ``geometric'' Gauss map for surfaces in $\E^3(\kappa,\tau)$ that coincides with the hyperbolic Gauss map \cite{fmajm} for graphs in $\h^2(\kappa)\times\R$ when $\tau=0$ and such that, by taking a limit when $\kappa\to0$, we get the Gauss map defined for surfaces in $\nil$ using the Lie group structure studied in \cite{danielimrn}.

For simplicity we set $\E=\E^3(\kappa,\tau)$. We denote by $\isomzero(\E)$ the connected component of the identity in the isometry group of $\E$. All isometries in $\isomzero(\E)$ leave the field $\xi$ invariant, since the sectional curvature of the plane $\xi^\perp$ is a strict minimum (see \cite{danielcmh} for more details); as a consequence, a fiber is mapped to another fiber. An isometry $f\in\isomzero(\E)$ induces via $\pi$ an orientation-preserving isometry $\tilde f$ of $\h^2(\kappa)$, i.e., $\tilde f\circ\pi=\pi\circ f$; we will say that $\tilde f$ is the \emph{horizontal part} of $f$.
Moreover, the field $\xi$ is a Killing field; the isometries it generates are called \emph{vertical translations} and their horizontal part is the identity.

Let
$$\mathrm{SU}_{1,1}=\left\{\left(
\begin{array}{cc}
\alpha & \beta \\
\bar\beta & \bar\alpha \end{array}\right);
(\alpha,\beta)\in\C^2,|\alpha|^2-|\beta|^2=1\right\},$$
$$\psu=\su/\{\pm\rmI\}.$$

We let $\cS$ denote the Riemann sphere $\bar\C=\C\cup\{\infty\}$ together with a marked oriented circle $\cE$ in $\bar\C$. We  call $\cE$ the \emph{equator} of $\cS$. Then $\cS\setminus\cE$ has two connected components: we will call \emph{northern hemisphere} and denote by $\cS^+$ the one whose oriented boundary is $\cE$, and we will call \emph{southern hemisphere} and denote by $\cS^-$ the one whose oriented boundary is $\cE$ with the opposite orientation.

The connected group that acts
naturally on $\cS$ is the group of (orientation-preserving) conformal diffeomorphisms of $\bar\C$ that leave $\cE$ invariant and preserve the
orientation of $\cE$. This group is $\psu\simeq\mathrm{PSL}_2(\R)$, which is also the group of orientation-preserving isometries of $\h^2(\kappa)$. Hence to each $M\in\su$ we associate a conformal diffeomorphism $\psi_M:\cS\to\cS$ as above and an orientation-preserving isometry
$\varphi_M:\h^2(\kappa)\to\h^2(\kappa)$, so that
$$\psi_{MN}=\psi_M\circ\psi_N,\quad\quad
\varphi_{MN}=\varphi_M\circ\varphi_N.$$ In this way we obtain a unique conformal diffeomorphism $\sigma:\h^2(\kappa)\cup\partial_\infty\h^2(\kappa)\to\cS^+\cup\cE$ such that
$$\sigma\circ\psi_M=\varphi_M\circ\sigma$$ for all $M$.

Without loss of generality we can assume that $\cE$ is the circle $|z|=1$
oriented counter clockwise, the northern hemisphere is $$\cS^+=\{|z|<1\},$$ the southern hemisphere is $$\cS^-=\{|z|>1\}\cup\{\infty\},$$ and, for $$M=\left(\begin{array}{cc}
\alpha & \beta \\
\bar\beta & \bar\alpha \end{array}\right)\in\mathrm{SU}_{1,1},$$
  $\psi_M$ and $\varphi_M$ are defined by
$$ \psi_M(z):=\frac{\alpha z+\beta}{\bar\beta z+\bar\alpha},\quad\quad
\varphi_M(\zeta):=\frac1c\frac{\alpha c\zeta+\beta}
{\bar\beta c\zeta+\bar\alpha}.$$
In this way, we have $\sigma(z)=z/c$.

Inspired by the properties of the ``hyperbolic Gauss map'' of surfaces in hyperbolic $3$-space $\h^3$, we seek a Gauss map for oriented surfaces in $\E$ with values in $\cS$ having the following properties:
\begin{itemize}
\item two surfaces passing through the same point $x$ have the same Gauss map at $x$ if and only if they have the same unit normal vector at $x$,
\item applying an isometry $f$ of $\E$ belonging to the connected component of the identity switches the Gauss map $g$ of the surface to $\psi_M\circ g$, where  $M\in\su$ corresponds to the horizontal part of $f$,
\item the Gauss map at a point lies on the equator $\cE$ (respectively, in the northern hemisphere) if and only if the unit normal
vector at that point is horizontal (respectively, upwards pointing),
\end{itemize}

The first property implies that the Gauss map of an oriented surface $\Sigma$ factorizes as $g=\Pi\circ N$ for some map $\Pi:\rmU\E\to\cS$, where $\rmU\E$ denotes the unit tangent bundle to $\E$ and $N:\Sigma\to\rmU\E$  the unit normal vector to $\Sigma$.

\begin{thm} \label{thm:gauss}
There exists a unique map 
$$\begin{array}{llll}
\Pi: & \rmU\E & \to & \cS \\
& (x,Z) & \mapsto & \Pi_x(Z)
\end{array}$$ such that
\begin{itemize}
\item[(a)] for any $f\in\isomzero(\E)$, if $M\in\su$ is such that the horizontal part of $f$ is $\varphi_M$, then
$$\Pi\circ\rmd f=\psi_M\circ\Pi,$$
\item[(b)] for any $(x,Z)\in\rmU\E$, $\Pi_x(Z)\in\cE$ (respectively $\Pi_x(Z)\in\cS^+$, $\Pi_x(Z)\in\cS^-$) if and only if $Z$ is horizontal (respectively, upwards pointing, downwards pointing),
\item[(c)] for any point $x\in\E$ and any horizontal vector $Z\in\rmU_x\E$, $\Pi_x(Z)=\sigma(p)$ where $p$ denotes the endpoint of the oriented geodesic of $\h^2(\kappa)$ passing through $\pi(x)$ with $\rmd\pi(Z)$ as tangent vector,
\item[(d)] for any point $x\in\E$, the map $\Pi_x:\rmU_x\E\to\cS$ is a conformal diffeomorphism.
\end{itemize}

Moreover the map $\Pi$ is analytic, and its expression in the frame $(V_1,V_2,V_3)$ is given by
\begin{equation} \label{formulaPi}
\Pi_x(Z)=\frac{Z_1+iZ_2+c\zeta(1+Z_3)}{c\bar\zeta(Z_1+iZ_2)+1+Z_3}
\end{equation}
where $\zeta=x_1+ix_2$ and $Z=Z_1V_1+Z_2V_2+Z_3V_3$.
\end{thm}

\begin{proof}
We first notice that proving existence and uniqueness of $\Pi$ is equivalent to proving that there exists a unique map $\Pi:\rmU\E\to\cS$ satisfying property (a) and properties (b), (c) and (d) at one point, for instance at the origin $O=(0,0,0)\in\E$.

We first prove the uniqueness of $\Pi_O$. At the origin the frame $(V_1,V_2,V_3)$ is just the coordinate frame. We let $s:\rmU_O\E\to\cS$ denote the stereographic projection with respect to the South Pole, i.e., if $Z=Z_1V_1+Z_2V_2+Z_3V_3\in\rmU_O\E$, then $$s(Z)=\frac{Z_1+iZ_2}{1+Z_3}.$$ So, by property (d) there exists a conformal diffeomorphism $h:\cS\to\cS$ such that $\Pi_O=h\circ s$.

By property (b), $h(\cS^+)=\cS^+$. Moreover, property (a) holds for all isometries $f\in\isomzero(\E)$ preserving $O$. Applying this property to the vector $Z=V_3$, we obtain that for all rotations $R$ around $0\in\cS$ we have $h(0)=R(h(0))$, which implies that $h$ is a rotation around $0\in\cS$. Finally, using property (c) we conclude that $h$ is the identity, and hence $\Pi_O=s$.

Observe that it is easy to check that $\Pi_O=s$ satisfies (b), (c), (d) at $x=O$ and also (a) for all isometries $f\in\isomzero(\E)$ preserving $O$.

We now prove existence and uniqueness of $\Pi$, as well as the announced expression \eqref{formulaPi} for $\Pi$.

We consider a point $y=(y_1,y_2,y_3)\in\E$ and an isometry $f\in\isomzero(\E)$ such that $f(y)=O$. To $f$ is associated a matrix $M=\left(\begin{array}{cc}
\alpha & \beta \\
\bar\beta & \bar\alpha \end{array}\right)\in\su$ such that
$$\beta=-\alpha cw$$ where $w=y_1+iy_2$. Let $$Z=Z_1V_1(y)+Z_2V_2(y)+Z_3V_3(y)\in\rmU_y\E$$ and set
$$\rmd f(Z)=Y_1V_1(0)+Y_2V_2(0)+Y_3V_3(0).$$ Since isometries in $\isomzero(\E)$ preserve the field $\xi$, we have $\langle\rmd f(Z),\xi(O)\rangle=\langle Z,\xi(y)\rangle$, i.e., $Y_3=X_3$.
On the other hand, since $\varphi_M$ is holomorphic, using complex
notation in $\h^2(\kappa)$ we have $\rmd\varphi_M =\varphi_M' (\rmd x_1+i\rmd x_2)$.


Consequently, since $\Lambda(0)=1$ and $$\rmd\pi(\rmd
f(Z))=\rmd\varphi_M(\rmd\pi(Z)),$$ we get
$$Y_1+iY_2=\frac{\varphi'_M(w)}{\Lambda(w)}(Z_1+iZ_2)
=\frac\alpha{\bar\alpha}(Z_1+iZ_2).$$
From this we deduce that
$$\Pi_O(\rmd f(Z))=\frac{Y_1+iY_2}{1+Y_3}=\frac\alpha{\bar\alpha}\frac{Z_1+iZ_2}{1+Z_3}.$$
Then by property (a) we have
$$\Pi_y(Z)=\psi_M^{-1}(\Pi_O(\rmd f(Z)))
=\frac{\bar\alpha\Pi_O(\rmd f(Z))-\beta}{-\bar\beta\Pi_O(\rmd f(Z))+\alpha}
=\frac{Z_1+iZ_2+cw(1+Z_3)}{c\bar w(Z_1+iZ_2)+1+Z_3}.$$


Let us observe that this expression does not depend on the choice of the isometry $f$; this is because $\Pi_O$ satisfies (a) for all isometries in $\isomzero(\E)$ preserving $O$. This gives the uniqueness of $\Pi$ as well as the announced expression, from which we also deduce analyticity. Conversely, by construction this map satisfies (a), (b), (c) and (d).
\end{proof}

\begin{defn} \label{defn:gauss}
Let $\Sigma$ be an oriented surface in $\E$. Then the \emph{Gauss map} of $\Sigma$ is the map $\Pi\circ N:\Sigma\to\cS$ where $N$ is the unit normal vector to $\Sigma$.
\end{defn}

\begin{rem}
Though we have used a specific model for $\E$ to prove Theorem \ref{thm:gauss}, its statement about existence and uniqueness of $\Pi:\rmU\E\to\cS$ satisfying conditions (a)-(d) does not involve a model for $\E$ but only needs the choice of an isomorphism between the group of orientation-preserving isometries of $\h^2(\kappa)$ and the group of orientation-preserving conformal diffeomorphisms of $\cS$ that leave $\cE$ invariant and preserve the orientation of $\cE$. In particular, Definition \ref{defn:gauss} is also independent of the model chosen for $\E$.
\end{rem}

If we consider the specific model for $\E$ we are working with in this paper, we get the following expression for the Gauss map.

\begin{cor} \label{ginframe}
If  $N=N_1V_1+N_2V_2+N_3V_3$, then the Gauss map of $\Sigma$ is
\begin{equation} \label{eqn:ginframe}
 g=\frac{N_1+iN_2+c\zeta(1+N_3)}{c\bar\zeta(N_1+iN_2)+1+N_3}
\end{equation}
with $\zeta=x_1+ix_2$.
\end{cor}

\begin{rem}
In the case of $\h^2\times\R$, the hyperbolic Gauss map defined in \cite{fmajm} for CMC $1/2$ graphs can be written as $\Pi\circ N$ with $\Pi$ satisfying conditions (a)-(d) in Theorem \ref{thm:gauss}, so it coincides with $g$.
\end{rem}

\begin{rem}
The formula given in Corollary \ref{ginframe} extends for $\kappa=0$, and in this case we get the Gauss map studied in \cite{danielimrn}. Hence formula \eqref{eqn:ginframe} provides a unified treatment for Gauss maps in all manifolds $\E^3(\kappa,\tau)$ with $\kappa\leqslant0$.
\end{rem}

\begin{ex}
Let $\gamma$ be a curve of constant curvature $k$ in $\h^2(\kappa)$ and let $\Sigma=\pi^{-1}(\gamma)\subset\E$. Then, by property (b) of Theorem \ref{thm:gauss}, the image of the Gauss map of $\Sigma$ lies in is the equator $\cE$, since the normal to $\Sigma$ is horizontal. When $|k|>1$, it is the whole $\cE$; when $|k|=1$ (in which case $\Sigma$ is a horocylinder), it is a point in $\cE$ or its complement (depending on the orientation); when $|k|<1$, it is an open arc in $\cE$. More generally, the Gauss map of a surface $\Sigma\subset\E$ takes values in $\cE$ if and only if the normal to $\Sigma$ is everywhere horizontal, i.e., if and only if $\pi(\Sigma)\subset\h^2(\kappa)$ is a curve. 
\end{ex}

\begin{rem} \label{rem:graph}
By property (b) of Theorem \ref{thm:gauss}, the Gauss map of a surface $\Sigma\subset\E$ takes values in $\D$ if and only if $\Sigma$ is a nowhere vertical surface with upwards pointing normal. 
\end{rem}

\begin{rem} \label{rotation}
Let $r:\E\to\E$ be the map defined by $r(x_1,x_2,x_3)=(x_1,-x_2,-x_3)$. This map is an isometry of $\E$ (but does not lie in $\isomzero(\E)$); it is the rotation of angle $\pi$ around the $x_1$-axis, which is a horizontal geodesic. Let $x\in E$ and $X\in\rmU_x\E$. Then
$$\Pi_{r(x)}(\rmd r(X))=\frac1{\Pi_x(X)}.$$ Indeed, this follows from \eqref{formulaPi} with $(\zeta,X_1,X_2,X_3)$ replaced by $(\bar\zeta,X_1,-X_2,-X_3)$.
\end{rem}

\begin{rem} \label{rem:liegroup}
The homogeneous manifold $\psl$ is isometric to the universal cover of $\rmU\h^2(\kappa)$ endowed with a Sasaki metric (see \cite{danielcmh} for details). Via this identification, the Riemannian fibration $\pi$ is the canonical projection onto $\h^2(\kappa)$, and isometries of $\psl$ coming from the Lie group structure (i.e., left-multiplications) are of the form $(f,\rmd f)$ where $f$ is a direct isometry of $\h^2(\kappa)$. For example, among all screw motions around the $x_3$-axis in $\psl$, those coming from the Lie group structure are exactly those whose pitch has a particular value that depends on $\kappa$ and $\tau$. Vertical translations do not come from the Lie group structure.
\end{rem}


\section{Preliminary calculations} \label{sec:conformal}

This section is devoted to preliminary equations for conformal  immersions in $\E=\E^3(\kappa,\tau)$ where $$\kappa\leqslant0$$ and their Gauss map. 

\subsection{Expression of the connection of $\E$}

First we compute the connection $\hat\nabla$ of $\E$ in the frame $(V_1,V_2,V_3)$. We have
$$[V_1,V_2]=-\frac\kappa2x_2V_1+\frac\kappa2x_1V_2+2\tau V_3,\quad
[V_1,V_3]=[V_2,V_3]=0,$$
and so
\begin{equation} \label{connection}
\begin{array}{lll}
\hat\nabla_{V_1}V_1=\frac\kappa2x_2V_2, &
\hat\nabla_{V_2}V_1=-\frac\kappa2x_1V_2-\tau V_3, &
\hat\nabla_{V_3}V_1=-\tau V_2, \\
\hat\nabla_{V_1}V_2=-\frac\kappa2x_2V_1+\tau V_3, &
\hat\nabla_{V_2}V_2=\frac\kappa2x_1V_1, &
\hat\nabla_{V_3}V_2=\tau V_1, \\
\hat\nabla_{V_1}V_3=-\tau V_2, &
\hat\nabla_{V_2}V_3=\tau V_1, &
\hat\nabla_{V_3}V_3=0.
\end{array}
\end{equation}

\subsection{Conformal  immersions} \label{sec:immersions}

Let $X=(x_1,x_2,x_3):\Sigma\to\E$ be a conformal immersion from a Riemann surface $\Sigma$ into $\E$. We shall denote by $N:\Sigma\to\rmU\E$ its unit normal.
If we fix a conformal coordinate $z=u+iv$ in $\Sigma$, then we have $$\langle X_z,X_{\bar z}\rangle=\frac{\lambda}2>0, \quad \langle X_z,X_z\rangle=0,$$ where $\lambda$ is the conformal factor of the metric with respecto to $z$. Moreover, we will denote the coordinates of $X_z$ and $N$ in the  frame $(V_1,V_2,V_3)$ by
$$X_z=\left[\begin{array}{c} A_1 \\ A_2 \\ A_3 \end{array}\right],\quad
N=\left[\begin{array}{c} N_1 \\ N_2 \\ N_3 \end{array}\right].$$
The function $N_3$ is the angle function (see Section \ref{sec:ekappatau}).

The usual \emph{Hopf differential} of $X$, i.e., the $(2,0)$ part of its complexified second fundamental form, is defined as $$P\rmd z^2=\langle N,\hat\nabla_{X_z}{X_z}\rangle\rmd z^2.$$
From the definitions, we have the basic algebraic relations
\begin{equation} \label{sumAk}
\left\{\def\arraystretch{1.4}\begin{array}{l}
|A_1|^2+|A_2|^2+|A_3|^2=\displaystyle \frac{\lambda}2,\quad \\
A_1^2+A_2^2+A_3^2=0,\quad \\
N_1^2+N_2^2+N_3^2=1,\quad \\
A_1N_1+A_2N_2+A_3N_3=0.
\end{array}\right.
\end{equation}
A classical computation proves that the Gauss-Weingarten equations of the immersion read as
\begin{equation*}
\left\{\def\arraystretch{1.9}\begin{array}{l}
\hat\nabla_{X_z}X_z=\displaystyle \frac{\lambda_z}{\lambda}X_z+PN,\quad \\
\hat\nabla_{X_{\bar z}}X_z= \displaystyle \frac{\lambda H}2N,\quad \\
\hat\nabla_{X_{\bar z}}X_{\bar z}=\displaystyle\frac{\lambda_{\bar z}}{\lambda}X_{\bar z}+\bar PN, \\
 \end{array}\right. \hspace{1cm}
\left\{\def\arraystretch{1.9}\begin{array}{l}
\hat\nabla_{X_z}N=-HX_z-\displaystyle\frac{2P}{\lambda}X_{\bar z},\quad \\
\hat\nabla_{X_{\bar z}}N=-\displaystyle\frac{2\bar P}{\lambda}X_z-HX_{\bar z}, \\
 \end{array}\right.
\end{equation*}
where $H$ is the mean curvature function of $X$.

Using \eqref{connection} in these equations we get
\begin{equation} \label{diffAkz}
\left\{\def\arraystretch{1.9} \begin{array}{l}
\displaystyle{A_{1\bar z}=\frac{\lambda H}2N_1+\frac\kappa2x_2\bar A_1A_2
-\frac\kappa2x_1|A_2|^2-\tau(\bar A_2A_3+A_2\bar A_3),} \\
\displaystyle{A_{2\bar z}=\frac{\lambda H}2N_2+\frac\kappa2x_1A_1\bar A_2
-\frac\kappa2x_2|A_1|^2+\tau(\bar A_1A_3+A_1\bar A_3),} \\
\displaystyle{A_{3\bar z}=\frac{\lambda H}2N_3+\tau(A_1\bar A_2-\bar A_1A_2).}
\end{array}\right.
\end{equation}
\begin{equation} \label{N3z}
N_{3z}=-HA_3-\frac{2P}{\lambda}\bar A_3+\tau(A_2N_1-A_1N_2).
\end{equation}
Moreover, the fact that $X_z\times X_{\bar z}=i\frac{\lambda}2N$ implies that
\begin{equation*} 
N_1=-\frac{2i}{\lambda}(A_2\bar A_3-A_3\bar A_2),\quad
N_2=-\frac{2i}{\lambda}(A_3\bar A_1-A_1\bar A_3),\quad
N_3=-\frac{2i}{\lambda}(A_1\bar A_2-A_2\bar A_1).
\end{equation*}

\subsection{An auxiliary map}


We now introduce the auxiliary  map
$$G:=\frac{N_1+iN_2}{1+N_3},$$ so that
\begin{equation}\label{tres55}
N=\frac1{1+|G|^2}\left[\begin{array}{c} 2\re G \\ 2\im G \\ 1-|G|^2
          \end{array}\right].\end{equation}

\begin{rem} \label{rem:modulusG}
This map $G$ does not have a geometric meaning since it depends on the choice of the frame $(V_1,V_2,V_3)$; however $|G|$ has a geometric meaning, since it depends only on the angle function of the immersion. In particular, $X$ defines a local graph with upwards pointing normal if and only if $|G|<1$.
\end{rem}

Let us set $$\eta:=2\langle X_z,\xi\rangle = 2 A_3.$$ A straightforward computation from \eqref{sumAk} proves that $\eta/\bar G$ extends smoothly at points where $G=0$ and that
\begin{equation} \label{asubis}
 A_1=-\frac{1-\bar G^2}{4\bar G}\eta,\quad
A_2=i\frac{1+\bar G^2}{4\bar G}\eta,\quad A_3=\frac{\eta}2
 \end{equation} and thereby
\begin{equation} \label{landa}
 \lambda=(1+|G|^2)^2\frac{|\eta|^2}{4|G|^2}.
  \end{equation}


 Then, denoting $\zeta=x_1+ix_2:\Sigma\to\C$,  \eqref{diffAkz} becomes
\begin{equation} \label{1line}
\begin{array}{l}
\displaystyle{(\bar G^2+1)\frac{\eta\bar G_{\bar z}}{4\bar G^2}+\frac14\left(\bar G-\frac1{\bar G}\right)\eta_{\bar z}} \\
\quad
=\displaystyle{(1+|G|^2)\frac{|\eta|^2}{4|G|^2}H\re(G)
-\kappa\frac{|\eta|^2}{32|G|^2}(1+\bar G^2)(\zeta+G^2\bar\zeta)
+i\tau\frac{|\eta|^2}8\left(\frac1G-\frac1{\bar G}+G-\bar G\right),}
\end{array}\end{equation}
\begin{equation} \label{2line}
\begin{array}{l}
\displaystyle{i(\bar G^2-1)\frac{\eta\bar G_{\bar z}}{4\bar G^2}+\frac i4\left(\bar G+\frac1{\bar G}\right)\eta_{\bar z}} \\
\quad
=\displaystyle{(1+|G|^2)\frac{|\eta|^2}{4|G|^2}H\im(G)
+i\kappa\frac{|\eta|^2}{32|G|^2}(1-\bar G^2)(\zeta+G^2\bar\zeta)
+\tau\frac{|\eta|^2}8\left(-\frac1G-\frac1{\bar G}+G+\bar G\right),}
\end{array}\end{equation}
\begin{equation} \label{3line}
\frac12\eta_{\bar z}=(H+i\tau)\frac{|\eta|^2}{8|G|^2}(1-|G|^4).\end{equation}
Reporting \eqref{3line} into \eqref{1line} + $i$ \eqref{2line} gives
$$\frac{G\bar G_{\bar z}}{2\bar\eta}=\frac H8(1+|G|^2)^2+\frac{i\tau}8(1-|G|^2)^2
-\frac\kappa{16}\bar G(\zeta+G^2\bar\zeta),$$
i.e.
\begin{equation} \label{formulaeta}
\eta=\frac{4\bar GG_z}{U(G,\zeta)}
\end{equation}
where
$$U(G,\zeta)=H(1+|G|^2)^2-i\tau(1-|G|^2)^2+2c^2G(\bar\zeta+\bar G^2\zeta).$$

\subsection{Preliminary formulas for the Gauss map}

Using \eqref{asubis} and \eqref{defframeV} we first have
\begin{equation} \label{zetaz}
\zeta_z=\frac{A_1+iA_2}{\Lambda(\zeta)}=-(1-c^2|\zeta|^2)\frac\eta{2\bar G},
\end{equation}
\begin{equation} \label{barzetaz}
\bar\zeta_z=\frac{A_1-iA_2}{\Lambda(\zeta)}=(1-c^2|\zeta|^2)\frac{\bar G\eta}2=-\bar G^2\zeta_z.
\end{equation}

By Corollary \ref{ginframe}, the Gauss map $g$ and the auxiliary map $G$ are related by
\begin{equation} \label{gG}
g=\frac{G+c\zeta}{c\bar\zeta G+1},\quad\quad G=\frac{g-c\zeta}{-c\bar\zeta g+1},
\end{equation}
from where we also get
\begin{equation} \label{modulusG}
 1-|G|^2=\frac{(1-c^2|\zeta|^2)(1-|g|^2)}{(1-c\bar\zeta g)(1-c\zeta\bar g)}.
\end{equation}
\begin{equation} \label{zetaG}
\bar\zeta+\bar G^2\zeta=
\frac{(\bar\zeta+\bar g^2\zeta)(1+c^2|\zeta|^2)-4c|\zeta|^2\bar g}{(1-c\zeta\bar g)^2}.
\end{equation}

\subsection{Preliminary formulas for the derivatives of the Gauss map}

Differentiating the left equation in \eqref{gG} we get
\begin{eqnarray*}
\frac{g_z}g & = & \frac{G_z+c\zeta_z}{G+c\zeta}-c\frac{\bar\zeta_zG+\bar\zeta G_z}{c\bar\zeta G+1} \\
& = & \frac{1-c^2|\zeta|^2}{(G+c\zeta)(c\bar\zeta G+1)}G_z+c\frac{\zeta_z}{G+c\zeta}
-c\frac{G\bar\zeta_z}{c\bar\zeta G+1}.
\end{eqnarray*}
Using \eqref{zetaz}, \eqref{barzetaz} and \eqref{formulaeta} we obtain
\begin{eqnarray*}
\frac{g_z}g & = & \frac{1-c^2|\zeta|^2}{(G+c\zeta)(c\bar\zeta G+1)}\frac{\eta U(G,\zeta)}{4\bar G}
-c(1-c^2|\zeta|^2)\left(\frac\eta{2\bar G(G+c\zeta)}+\frac{|G|^2\eta}{2(c\bar\zeta G+1)}\right) \\
& = & \frac{(1-c^2|\zeta|^2)\eta}{4\bar G(G+c\zeta)(c\bar\zeta G+1)}
(U(G,\zeta)-2c(c\bar\zeta G+1)-2cG\bar G^2(G+c\zeta)) \\
& = & \frac{(1-c^2|\zeta|^2)\eta}{4\bar G(G+c\zeta)(c\bar\zeta G+1)}V(G)
\end{eqnarray*}
with
\begin{equation} \label{defV}
V(G)=(H-c)(1+|G|^2)^2-(c+i\tau)(1-|G|^2)^2.
\end{equation}

On the other hand, differentiating the right equation in \eqref{gG} and using \eqref{zetaz} and \eqref{barzetaz} we get
\begin{equation} \label{diffbarG}
\frac{\bar G_z}{\bar G}=\frac{1-c^2|\zeta|^2}{(\bar g-c\bar\zeta)(1-c\zeta\bar g)}
\left(\bar g_z-c\eta\bar g+\frac{c^2\eta}2(\bar\zeta+\bar g^2\zeta)\right).
\end{equation}

\section{CMC surfaces with critical mean curvature} \label{sec:critical}

We recall that there exist compact CMC $H$ surfaces in $\E$ if and only if $|H|>c$. Indeed, there exist CMC $H$ spheres if $|H|>c$ (see for instance \cite{ar1,torralbo,mercuri}), and since horocylinders (vertical planes when $\kappa=0$) have mean curvature $c$, by the maximum principle there cannot exist compact CMC $H$ surfaces if $|H|\leqslant c$. CMC surfaces with mean curvature $c$ are called \emph{CMC surfaces with critical mean curvature}.

In this section we focus on CMC surfaces with critical mean curvature, and we assume that they are oriented by their mean curvature vector when $\kappa<0$, i.e., we set $H=c$. The results of these section generalize those of \cite{danielimrn} (which correspond to the case where $\kappa=0$ and where $\tau$ has been normalized to $1/2$).

\subsection{Harmonicity of the Gauss map}

\begin{thm} \label{rep}
Let $X=(x_1,x_2,x_3):\Sigma\to\E$ be a CMC immersion  with critical mean curvature. Assume that $X$ is nowhere vertical and with upwards pointing unit normal vector. Let $g:\Sigma\to\D$ be its Gauss map (see Remark \ref{rem:graph}). Then, $g$ is nowhere antiholomorphic, i.e., $g_z\neq 0$ at every point for any local conformal parameter $z$ on $\Sigma$, and $g$ verifies the second order elliptic equation
\begin{equation} \label{eqg2}
(1-|g|^2)g_{z\bar z}+2\bar gg_z\bar g_{\bar z}=0.
\end{equation}
In other words, $g$ is harmonic in $\D$ endowed with the hyperbolic metric (of curvature $-1$) $4|\rmd w|^2/(1-|w|^2)^2$.

Moreover, the immersion $X=(x_1,x_2,x_3):\Sigma\to\E$ can be recovered in terms of the Gauss map $g$ by means of the representation formula
\begin{equation} \label{repfor}
\left\{\begin{array}{lll}
\zeta_z & = & \displaystyle{\frac2{c+i\tau}\frac{(1-c\zeta\bar g)^2}{(1-|g|^2)^2}g_z,} \\
\zeta_{\bar z} & = &
\displaystyle{-\frac2{c-i\tau}\frac{(g-c\zeta)^2}{(1-|g|^2)^2}\bar g_{\bar z},} \\
(x_3)_z & = & \displaystyle{-\frac2{c+i\tau}\frac{(\bar g-c\bar\zeta)(1-c\zeta\bar g)}{(1-c^2|\zeta|^2)(1-|g|^2)^2}g_z
+\frac{i\tau}2\frac{\zeta\bar\zeta_z-\bar\zeta\zeta_z}{1-c^2|\zeta|^2},}
\end{array}\right.
\end{equation}
where $$\zeta=x_1+ix_2.$$

Conversely, assume that a map $g:\Sigma\to\D$ from a simply
connected Riemann surface $\Sigma$ verifies \eqref{eqg2} and is
nowhere antiholomorphic. Then the map $X:\Sigma\to\E$ given by the
representation formula \eqref{repfor} is a  conformal CMC immersion
with critical mean curvature whose Gauss map is $g$.
\end{thm}

\begin{proof}
Since $H=c$, $V(G)$ defined in \eqref{defV} simplifies and we obtain, using \eqref{modulusG} and the second equation in \eqref{gG},
\begin{equation} \label{gzeta}
g_z=-\frac{c+i\tau}4\frac{(1-c^2|\zeta|^2)(1-|g|^2)^2}{(\bar g-c\bar\zeta)(1-c\zeta\bar g)}\eta.
\end{equation}

 Differentiating \eqref{gzeta}, we get
\begin{eqnarray*}
\frac{g_{z\bar z}}{g_z} & = & \frac{\eta_{\bar z}}\eta
-c^2\frac{\bar\zeta\zeta_{\bar z}+\zeta\bar\zeta_{\bar z}}{1-c^2|\zeta|^2}
-2\frac{\bar gg_{\bar z}+g\bar g_{\bar z}}{1-|g|^2}
-\frac{\bar g_{\bar z}-c\bar\zeta_{\bar z}}{\bar g-c\bar\zeta}
+c\frac{\bar g\zeta_{\bar z}+\zeta\bar g_{\bar z}}{1-c\zeta\bar g} \\
& = & (c+i\tau)\frac{1-|G|^4}{4|G|^2}\bar\eta
+c^2(\zeta-\bar\zeta G^2)\frac{\bar\eta}{2G}
-2\frac{\bar gg_{\bar z}}{1-|g|^2}-2\frac{g\bar g_{\bar z}}{1-|g|^2} \\
& &
-c\frac{1-c^2|\zeta|^2}{\bar g-c\bar\zeta}\frac{\bar\eta}{2G}
+c\frac{(1-c^2|\zeta|^2)\bar g}{1-c\zeta\bar g}\frac{G\bar\eta}{2}
+\frac{-1-c^2|\zeta|^2+2c\zeta\bar g}{(\bar g-c\bar\zeta)(1-c\zeta\bar g)}\bar g_{\bar z}.
\end{eqnarray*}

We compute
\begin{eqnarray*}
B_1 & := & (c+i\tau)\frac{1-|G|^4}{4|G|^2}\bar\eta \\
& = & (c+i\tau)(1-c^2|\zeta|^2)(1-|g|^2)
\frac{(1-c\bar\zeta g)(1-c\zeta\bar g)+(g-c\zeta)(\bar g-c\bar\zeta)}
{4(g-c\zeta)(\bar g-c\bar\zeta)(1-c\bar\zeta g)(1-c\zeta\bar g)}\bar\eta, \\
B_2 & := & c^2(\zeta-\bar\zeta G^2)\frac{\bar\eta}{2G} \\
& = & c^2\frac{(\zeta-\bar\zeta g^2)(1-c^2|\zeta|^2)}{2(g-c\zeta)(1-c\bar\zeta g)}\bar\eta, \\
B_3 & := & -c\frac{1-c^2|\zeta|^2}{\bar g-c\bar\zeta}\frac{\bar\eta}{2G}
+c\frac{(1-c^2|\zeta|^2)\bar g}{1-c\zeta\bar g}\frac{G\bar\eta}{2} \\
& = & c(1-c^2|\zeta|^2)
\frac{c\bar g(\zeta-\bar\zeta g^2)(1-c^2|\zeta|^2)+c^2(\zeta^2\bar g^2-\bar\zeta^2g^2)
+2cg(\bar\zeta-\zeta\bar g^2)+|g|^4-1}
{2(g-c\zeta)(\bar g-c\bar\zeta)(1-c\bar\zeta g)(1-c\zeta\bar g)}\bar\eta,
\end{eqnarray*}
and so
$$B_1+B_2+B_3=-(c-i\tau)(1-c^2|\zeta|^2)(1-|g|^2)
\frac{(1-c\bar\zeta g)(1-c\zeta\bar g)+(g-c\zeta)(\bar g-c\bar\zeta)}
{4(g-c\zeta)(\bar g-c\bar\zeta)(1-c\bar\zeta g)(1-c\zeta\bar g)}\bar\eta.$$
Moreover, we have
$$-2\frac{g\bar g_{\bar z}}{1-|g|^2}
+\frac{-1-c^2|\zeta|^2+2c\zeta\bar g}{(\bar g-c\bar\zeta)(1-c\zeta\bar g)}\bar g_{\bar z}
=-\frac{(1-c\bar\zeta g)(1-c\zeta\bar g)+(g-c\zeta)(\bar g-c\bar\zeta)}
{(1-|g|^2)(\bar g-c\bar\zeta)(1-c\zeta\bar g)}\bar g_{\bar z},$$
and we conclude from \eqref{gzeta} that
$$\frac{g_{z\bar z}}{g_z}=-2\frac{\bar gg_{\bar z}}{1-|g|^2},$$
which means that $g:\Sigma\to\D$ is harmonic for the hyperbolic metric.

Combining \eqref{zetaz}, \eqref{barzetaz}, \eqref{gG} and \eqref{gzeta} gives the first order differential system satisfied by $\zeta$ in terms of $g$ and $g_z$:
\begin{equation} \label{syszeta}
\left\{\begin{array}{lll}
\zeta_z & = & \displaystyle{\frac2{c+i\tau}\frac{(1-c\zeta\bar g)^2}{(1-|g|^2)^2}g_z,} \\
\zeta_{\bar z} & = &
\displaystyle{-\frac2{c-i\tau}\frac{(g-c\zeta)^2}{(1-|g|^2)^2}\bar g_{\bar z}.}
       \end{array}\right.
\end{equation}
Observe that  \eqref{gzeta} gives $\eta$ in terms of $g$, $g_z$ and $\zeta$:
 \begin{equation}\label{star20}
\eta=-\frac4{c+i\tau}\frac{(\bar g-c\bar\zeta)(1-c\zeta\bar
g)}{(1-c^2|\zeta|^2)(1-|g|^2)^2}g_z.
 \end{equation}
Finally, since
$x_{3z}=A_3+\tau(x_1A_2-x_2A_1)=A_3+\tau\Lambda(x_1x_{2z}-x_2x_{1z})$,
we have
\begin{equation} \label{x3z}
 (x_3)_z=\frac\eta2+\frac{i\tau}2\frac{\zeta\bar\zeta_z-\bar\zeta\zeta_z}{1-c^2|\zeta|^2}.
\end{equation}

Moreover, if $g_z=0$ at some point, then these formulas show that $\zeta_z=0$, $\zeta_{\bar z}=0$ and $(x_3)_z=0$ at this point, contradicting the fact that $X$ is an immersion. Consequently, $g$ is nowhere antiholomorphic. We observe that the conformal factor of $X$ is given in terms of $g$ as
$$\lambda=\frac{4\left(|1-c\bar\zeta g|^2+|g-c\zeta|^2\right)^2}{(c^2+\tau^2)(1-c^2|\zeta|^2)^2(1-|g|^2)^4}.$$  

Conversely, assume that a map $g:\Sigma\to\D$ from a simply
connected Riemann surface $\Sigma$ verifies \eqref{eqg2} and is
nowhere antiholomorphic. A long but straightforward computation
shows that if $z$ denotes a complex parameter on $\Sigma$, then
$$\left(\displaystyle{\frac2{c+i\tau}\frac{(1-c\zeta\bar
g)^2}{(1-|g|^2)^2}g_z}\right)_{\bar{z}}\ =
\left(\displaystyle{-\frac2{c-i\tau}\frac{(g-c\zeta)^2}{(1-|g|^2)^2}\bar
g_{\bar z}}\right)_z.$$ By the Frobenius theorem, given $p_1+ip_2\in
\D(1/c)$ and $z_0\in \Sigma$, there is a unique solution $\zeta
:\Sigma\flecha \C$ to \eqref{syszeta} such that $\zeta (z_0)=p_1+ip_2$.
Let us check now that $|\zeta|<1/c$ on $\Sigma$.

Let $\phi:=1-c^2 |\zeta|^2$. Then, from \eqref{syszeta} we get

\begin{equation}\label{star10}
\phi_z = \frac{-2c^2 (\bar{\zeta} - \bar{g}^2
\zeta)}{(c+i\tau)(1-|g|^2)^2} g_z \phi.
\end{equation}
Assume that $|\zeta|<1/c$ does not hold in $\Sigma$ (observe that
$|\zeta(z_0)|<1/c$). Then there exists a regular path
$\gamma(t):[0,1]\flecha \Sigma$ such that $(\phi\circ \gamma)(0)>0$,
$(\phi\circ \gamma)(1)=0$ and $(\phi\circ \gamma)'(t)<0$ for all
$t\in [0,1)$. As $\gamma[0,1]$ is compact, we get from
\eqref{star10} that $$-\frac{(\phi\circ \gamma)'(t)}{(\phi\circ
\gamma)(t)} \leqslant C$$ for every $t\in [0,1)$ and some constant $C>0$.
Denote $f(t)=\log((\phi\circ \gamma)(t)) - \log((\phi\circ
\gamma)(0)).$ Then $f(0)=0$ and $f'(t)\geqslant -C$ for every
$t\in [0,1)$. Thus $f(t)\geqslant -C t$, which contradicts that
$\lim_{t\to1^-}f(t)=-\8$. So, $|\zeta|<1/c$ on $\Sigma$.

Once here, a computation shows that, if we write the third equation
in \eqref{repfor} as $(x_3)_z=:\cA$, then
$\cA_{\bar{z}}=\bar{\cA}_z$. Again by the Frobenius theorem there
exists $x_3:\Sigma\flecha \R$ such that $(x_3)_z=\cA$, unique once
we fix an initial condition $x_3 (z_0)=p_3$.

Therefore, we get a unique map $X=(x_1,x_2,x_3):\Sigma\flecha
\mathbb{E}$ with $X(z_0)=(p_1,p_2,p_3)$ that verifies system
\eqref{repfor}, where $\zeta =x_1+ix_2$.

To complete the proof of the converse, it remains to check that $X$
is a conformal immersion with critical mean curvature $c$ and Gauss
map $g$. This is a long but standard computation. The idea is,
starting with \eqref{repfor}, to show that if $X_z=A_1 V_1 + A_2 V_2
+ A_3 V_3$, then $A_1,A_2$ are given by \eqref{zetaz},
\eqref{barzetaz} and $A_3=:2 \eta$ is given by \eqref{star20}. From
here we can check that $X$ is a conformal immersion with Gauss map
$g$, and a computation using \eqref{defV} gives that $X$ has
critical constant mean curvature $c$. We omit the details.

\end{proof}

\begin{rem} \label{generalcase}
If $X:\Sigma\to\E$ is a CMC immersion with critical mean curvature possibly with vertical points, then its Gauss map $g:\Sigma\to\bar\C$ still satisfies \eqref{eqg2}, and it is nowhere antiholomorphic at points where $|g|\neq1$. Indeed, the whole proof can be extended to this case.

However, conversely, if $\Sigma$ is simply connected and $g:\Sigma\to\bar\C$ satisfies \eqref{eqg2} without assuming that $|g|<1$, then the representation formula \eqref{repfor} only locally defines a possibly branched immersion away from points where $|g|=1$ (see Remark 4.2 in \cite{danielimrn}).
\end{rem}

\begin{rem}\label{changeor}
When $\kappa<0$, if $X$ is a nowhere vertical CMC immersion with
critical mean curvature, it may not be possible to orient it so that
its unit normal is upwards pointing, since the orientation is chosen so that
the mean curvature is positive. However it is not a loss of
generality to assume that the unit normal is upwards pointing. Indeed, assume
that $X$ has a downwards pointing unit normal and let $r$ the rotation of
angle $\pi$ around the $x_1$-axis (see Remark \ref{rotation}). This
map $r$ is an isometry of $\E$, and hence $r\circ X$ is a nowhere
vertical CMC immersion with critical mean curvature $c$ and upwards pointing
normal. Then Theorem \ref{rep} applies to $r\circ X$. Also, if $g$
denotes the Gauss map of $X$, then the Gauss map of $r\circ X$ is
$1/g$.
\end{rem}

\begin{rem}
Theorem \ref{rep} implies that, when $\kappa<0$ (i.e., in
$\h^2\times\R$ and $\psl$), there generically exists a two-parameter
family of non-congruent CMC immersions with critical mean curvature
sharing the same Gauss map. Indeed, in Theorem \ref{rep}, for a
given $z_0\in\Sigma$ one can prescribe $\zeta(z_0)$ and $x_3(z_0)$;
changing $x_3(z_0)$ simply corresponds to a vertical translation,
but changing $\zeta(z_0)$ gives a new non-congruent immersion unless
the Gauss map has some symmetry, by property (a) in Theorem
\ref{thm:gauss}. In $\h^2\times\R$ this was noticed in Remark 5 in
\cite{fmajm}. When $\kappa=0$ (i.e., in $\nil$) this is not the case
anymore (see Section 6 in \cite{danielimrn}).
\end{rem}

\subsection{The Hopf differential} \label{sec:hopf}

One can associate two natural holomorphic quadratic differentials to the immersion $X$:
\begin{itemize}
 \item the Hopf differential $Q(g)\rmd z^2$ of $g$ (in the sense of harmonic maps into $\D$ endowed with the hyperbolic metric):
$$Q(g):=\frac{4g_z\bar g_z}{(1-|g|^2)^2},$$
\item the Abresch-Rosenberg differential $\Phi\rmd z^2$ of $X$ \cite{ar1,ar2} (see also \cite{fmdga} for its expression):
$$\Phi:=2(H+i\tau)P-(\kappa-4\tau^2)\frac{\eta^2}4=2(c+i\tau)P+(c^2+\tau^2)\eta^2.$$
\end{itemize}

\begin{prop} \label{hopfar}
We have $$Q(g)=-\Phi.$$
\end{prop}

\begin{proof}
By \eqref{gzeta} we have
\begin{equation*}
Q(g)=-(c+i\tau)\frac{(1-c^2|\zeta|^2)}{(\bar g-c\bar\zeta)(1-c\zeta\bar g)}\eta\bar g_z.
\end{equation*}

We now compute $\Phi$. First we compute $P$, starting from the identity \eqref{N3z} and using \eqref{tres55}, \eqref{asubis}, \eqref{landa} and \eqref{formulaeta}:
\begin{eqnarray*}
P & = & \frac\lambda{\bar\eta}\left(2\frac{\bar GG_z+G\bar G_z}{(1+|G|^2)^2}-(H-i\tau)\frac\eta2\right) \\
& = & \frac\eta{2|G|^2}\left(\frac{U(G,\zeta)\eta}4+G\bar G_z\right)
-(H-i\tau)(1+|G|^2)^2\frac{\eta^2}{8|G|^2} \\
& = & \frac{i\tau}2\eta^2+\frac{c^2\eta^2}{4\bar G}(\bar\zeta+\bar G^2\zeta)+\frac{\bar G_z\eta}{2\bar G}.
\end{eqnarray*}
Then, using \eqref{diffbarG} and \eqref{zetaG}, we get
\begin{eqnarray*}
P & = & \frac{i\tau}2\eta^2
+\frac{c^2\eta^2}2\frac{\bar\zeta+\bar g^2\zeta}{(\bar g-c\bar\zeta)(1-c\zeta\bar g)}
-\frac{c\bar g\eta^2}2\frac{1+c^2|\zeta|^2}{(\bar g-c\bar\zeta)(1-c\zeta\bar g)}
+\frac\eta2\frac{1-c^2|\zeta|^2}{(\bar g-c\bar\zeta)(1-c\zeta\bar g)}\bar g_z \\
& = & -\frac{c-i\tau}2\eta^2+\frac\eta2\frac{1-c^2|\zeta|^2}{(\bar g-c\bar\zeta)(1-c\zeta\bar g)}\bar g_z.
\end{eqnarray*}
Finally we get $\Phi=-Q(g)$.
\end{proof}

\subsection{The Gauss map of the sister minimal surface in $\nil(\hat\tau)$} \label{sec:sister}

Let $X:\Sigma\flecha \E^3(\kappa,\tau)$ be a CMC immersion 
 with critical mean
curvature. Assume that $\Sigma$ is simply connected. Then, by the Lawson-type correspondence in
\cite{danielcmh}, there exists a minimal immersion
$\hat{X}:\Sigma\flecha \nil(\hat\tau)=\E^3(0,\hat\tau)$ where
$\hat{\tau}^2 = \tau^2 +c^2$ that is isometric to $X$. Moreover, the
immersions $X$ and $\hat{X}$ have the same angle function (see Section
\ref{sec:ekappatau}) and their associated Abresch-Rosenberg
differentials $\Phi$ and $\hat{\Phi}$ are related by $\Phi =
e^{-2i\theta}\hat{\Phi}$, where $\tau + i c = e^{i\theta}
\hat{\tau}$ (see for instance the end of Section 3 in \cite{gmm} and the expression of the Abresch-Rosenberg differential given in Section \ref{sec:hopf}).

A pair of such isometric surfaces are called \emph{sister surfaces}.
The minimal sister surface in $\nil(\hat\tau)$ of a CMC surface in
$\E^3(\kappa,\tau)$ with critical mean curvature is unique up to
direct isometries of $\nil(\hat\tau)$.

If $g:\Sigma\to\D$ is a harmonic map for the hyperbolic metric on $\D$, we consider the form
$\mu(g)|\rmd z|^2$ (which is independent of the choice of the
conformal parameter $z$) where
$$\mu(g):=\frac{4(|g_z|^2+|g_{\bar z}|^2)}{(1-|g|^2)^2}.$$ Two harmonic maps $g_1$ and $g_2$ are said to be \emph{associate} if there exists a real constant $\theta$ such that $Q(g_2)\rmd z^2=e^{-2i\theta}Q(g_1)\rmd z^2$ and $\mu(g_2)|\rmd z|^2=\mu(g_1)|\rmd z|^2$.


We now consider a  conformal immersion $X:\Sigma\to\E^3(\kappa,\tau)$ from a simply connected Riemann surface $\Sigma$ into $\E^3(\kappa,\tau)$ such that $X(\Sigma)$ is a local graph with positive angle function. We let  $\hat X:\Sigma\to\nil(\hat\tau)$ be its sister minimal immersion.  Let $g$ be the Gauss map of $X$ and $\hat g$ the Gauss map of $\hat X$ (which is unique only up to rotations around $0\in\D$, see Section 6 in \cite{danielimrn}).


\begin{prop}
The Gauss maps  $g$ and $\hat g$ are associate and $Q(g)=e^{-2i\theta}Q(\hat g)$  with $\theta$ defined by 
\begin{equation} \label{deftheta}
\tau+ic=e^{i\theta}\hat\tau.
\end{equation}
\end{prop}

\begin{proof}
We keep for the immersion $X$ the same notation as in the previous sections.

By Proposition \ref{hopfar} we have $Q(g)=-\Phi$, and similarly we have $Q(\hat g)=-\hat\Phi$ where $\hat\Phi$ is the Abresch-Rosenberg differential of $\hat X$. Since the differentials $\Phi$ and $\hat\Phi$ satisfy $\Phi=e^{-2i\theta}\hat\Phi$, we get $Q(g)=e^{-2i\theta}Q(\hat g)$.

We now prove that $\mu(g)=\mu(\hat g)$.


Since sister surfaces have the same angle function, it follows from \eqref{tres55} and the analogous formula in $\nil(\hat\tau)$ (see Definition 3.1 in \cite{danielimrn}) that
\begin{equation} \label{Ghatg}
 |G|=|\hat g|.
\end{equation}

Denote by $(\hat g,\hat\eta)$ the Weierstrass data of $\hat X$ (see Defintion 3.2 in \cite{danielimrn}). By equation (10) in \cite{danielimrn} (or by \eqref{formulaeta} with $(\eta,H,\tau,c)$ replaced by $(\hat\eta,0,\hat\tau,0)$)
we have
\begin{equation} \label{hateta}
 \hat\eta=\frac{4i}{\hat\tau}\frac{\overline{\hat g}\hat g_z}{(1-|\hat g|^2)^2}.
\end{equation}
On the other hand we have (see for instance the end of Section 3 in \cite{gmm})
\begin{equation} \label{etahateta}
 \eta=e^{-i\theta}\hat\eta.
\end{equation}
So using \eqref{gzeta} we get
$$g_z=\frac{\tau-ic}{\tau+ic}\frac{(1-c^2|\zeta|^2)(1-|g|^2)^2}{(\bar g-c\bar\zeta)(1-c\zeta\bar g)}
\frac{\overline{\hat g}\hat g_z}{(1-|\hat g|^2)^2}$$
and so, using \eqref{deftheta}, \eqref{Ghatg} and \eqref{modulusG},
\begin{equation} \label{ggammazeta}
\frac{g_z}{1-|g|^2}=
e^{-2i\theta}\frac{\hat g_z}{1-|\hat g|^2}\frac{\overline{\hat g}(1-c\bar\zeta g)}{\bar g-c\bar\zeta},
\end{equation}
which implies,  using \eqref{Ghatg} and \eqref{gG},
\begin{equation} \label{modulus1}
\frac{|g_z|}{1-|g|^2}=\frac{|\hat g_z|}{1-|\hat g|^2}.
\end{equation}
Then \eqref{modulus1} and the fact that $|Q(g)|=|Q(\hat g)|$ imply that
\begin{equation} \label{modulus2}
\frac{|g_{\bar z}|}{1-|g|^2}=\frac{|\hat g_{\bar z}|}{1-|\hat g|^2}.
\end{equation}
Finally, we deduce from \eqref{modulus1} and \eqref{modulus2} that $\mu(g)=\mu(\hat g)$. This proves that $g$ and $\hat g$ are associate.
\end{proof}

\begin{rem}
 Formula \eqref{ggammazeta} gives the expression of $\zeta$ in terms of $g$ and $\hat g$ without the need to integrate a differential system, contrarily to the representation formula \eqref{repfor}. We will use \eqref{ggammazeta} to explicitely compute Example \ref{ex1}.
\end{rem}

We recall the following result from \cite{korea} (see Corollaries
4.6.3 and 4.6.4 there; see also Theorem 6.10 in \cite{ICMProceedings}).

\begin{prop} \label{entire}
The following conditions are equivalent for a surface $X(\Sigma)$ of
critical CMC in $\E^3(\kappa,\tau)$:
 \begin{enumerate}
   \item[$(1)$]
 $X(\Sigma)$ is an entire graph,
\item[$(2)$]
$X(\Sigma)$ is a complete local graph.
 \end{enumerate}
In particular, the sister correspondence preserves entire graphs of
critical CMC (see Section \ref{sec:ekappatau} for the definition).
\end{prop}

\begin{proof}
That $(1)$ implies $(2)$ is trivial. Conversely, let $X(\Sigma)$ be
a complete local graph in $\E^3(\kappa,\tau)$, which we will assume
to be simply connected by passing to its universal covering if
necessary. Let $\rmd s^2$ denote the metric induced by $X$ on $\Sigma$
and $\nu$ the angle function of $X(\Sigma)$. Let $\hat{X}(\Sigma)$
denote the sister surface of $X$ in ${\rm Nil}_3$. Then, $\rmd s^2$ and
$\nu$ coincide with the metric induced by $\hat{X}$ on $\Sigma$, and
with the angle function of $\hat{X}(\Sigma)$, respectively. Then, in
particular, $\hat{X}(\Sigma)$ is a complete local graph in ${\rm
Nil}_3$. Consequently, by Theorem 3.1 in [7], $\hat{X}(\Sigma)$ is
an entire graph in ${\rm Nil}_3$. By Corollary 13 in \cite{fmtams}
(which uses a result from \cite{chengyau}), the metric $\nu^2\rmd s^2$ is
complete, and so by Lemma 9 in \cite{fmtams} the surface $X(\Sigma)$
is an entire graph in $\E^3(\kappa,\tau)$. This shows that $(2)$
implies $(1)$ and completes the proof.
\end{proof}

%
%
%
%
%

\begin{exa}[Surface with conformal Gauss map] \rm \label{ex1}
We seek a local graph with critical CMC whose Gauss map $g:\Sigma\to\D$ is conformal, i.e., such that the Hopf differential of $g$ vanishes identically. Then $\hat g$ is also conformal, so the surface will be the sister surface of that of Example 8.1 in \cite{danielimrn}. Up to a reparametrization we may assume that $\Sigma=\D$ and that $\hat g(z)=z$. Then up to an isometry of $\E^3(\kappa,\tau)$ we can assume that $g(z)=z$.

Then by \eqref{ggammazeta} we get
$$\zeta=\frac{e^{2i\theta}-1}c\frac z{e^{2i\theta}|z|^2-1}.$$ Using \eqref{hateta}, \eqref{etahateta} and \eqref{deftheta} we get
$$\eta=\frac{4i}{\tau+ic}\frac{\bar z}{(1-|z|^2)^2}.$$ Next, we compute that
$$\zeta_z=\frac{1-e^{2i\theta}}{c(e^{2i\theta}|z|^2-1)^2},\quad
\bar\zeta_z=\frac{(1-e^{-2i\theta})e^{-2i\theta}}c\frac{\bar z^2}{(e^{-2i\theta}|z|^2-1)^2},$$
$$1-c^2|\zeta|^2=\frac{(1-|z|^2)^2}{(e^{2i\theta}|z|^2-1)(e^{-2i\theta}|z|^2-1)}.$$
Then we compute using \eqref{x3z} that
$$(x_3)_z=\left(\frac{2i}{\tau+ic}
-\frac{i\tau(1-e^{2i\theta})(1-e^{-2i\theta})}{2c^2}
\frac{|z|^4-2|z|^2e^{-2i\theta}+1}{(e^{2i\theta}|z|^2-1)(e^{-2i\theta}|z|^2-1)}\right)\frac{\bar z}{(1-|z|^2)^2}.$$
By \eqref{deftheta} we have $(1-e^{2i\theta})(1-e^{-2i\theta})=4c^2/(\tau^2+c^2)$, so we get
\begin{eqnarray*}
 (x_3)_z & = & \frac2{\tau^2+c^2}
\frac{c(|z|^4+1)+2(\tau\sin(2\theta)-c\cos(2\theta))|z|^2}{(1-|z|^2)^2(e^{2i\theta}|z|^2-1)(e^{-2i\theta}|z|^2-1)}\bar z \\
& = & \frac{2c}{\tau^2+c^2}\frac{(1+|z|^2)^2\bar z}{(1-|z|^2)^2(|z|^4-2|z|^2\cos(2\theta)+1)},
\end{eqnarray*}
and finally, up to a vertical translation,
$$x_3=-\frac\tau{c^2}\arctan\left(\frac{|z|^2-\cos(2\theta)}{\sin(2\theta)}\right)+\frac2{c(1-|z|^2)}.$$
Hence we obtain a surface of revolution about the $x_3$-axis; this surface is also an entire graph and $x_3\to+\infty$ as $|\zeta|\to1/c$. This surface is described in \cite{torralbo} (item 2 of Theorem 2 with $4H^2+\kappa=0$ and $E=0$). In $\h^2(\kappa)\times\R$ (i.e., when $\tau=0$), this is the surface described on page 1172 in \cite{fmajm} (see also \cite{ar1}). 
\end{exa}


\begin{exa}[Surfaces with singular Gauss map] \rm \label{ex:singular}
We seek  complete local graphs with critical CMC whose Gauss map $g:\Sigma\to\D$ is singular on an open set of $\Sigma$. Then  $g(\Sigma)$ and $\hat g(\Sigma)$ are geodesics of $\D$ for the hyperbolic metric, and these surfaces will be sister surfaces of those of Example 8.2 in \cite{danielimrn}. These minimal surfaces in $\nil$ constitute a one-parameter family of surfaces that are not isometric one to another, the parameter being the hyperbolic distance between $\hat g(\Sigma)$ and the origin $0\in\D$. Since sister surfaces are isometric, we get a one-parameter family of CMC surfaces in $\E^3(\kappa,\tau)$ that are not isometric one to another. Since their sister minimal surfaces in $\nil$ are invariant by a one-parameter family of translations, a standard argument using the integrability equations for surfaces in $\E^3(\kappa,\tau)$ \cite{danielcmh} shows that these surfaces are invariant by a one-parameter family of isometries of $\E^3(\kappa,\tau)$ (see Proposition 4.3 in \cite{gmm}). Moreover, since their sister minimal surfaces in $\nil$ are entire graphs, it follows from Proposition \ref{entire} that these surfaces are also entire graphs. In $\h^2(\kappa)\times\R$ (i.e., when $\tau=0$), these are the surfaces described in Proposition 18 in \cite{fmajm} (see also \cite{saearp}). 
\end{exa}

\section{The Gauss map in the Lorentzian model} \label{sec:lorentzian}

In this section we assume that $\kappa<0$ and we will show an alternative expression for the Gauss map of nowhere vertical surfaces  when we consider the Lorentzian model for the hyperbolic space, that is, we now consider $\pi:\E\to\h^2(\kappa)$, where 
$$\h^2(\kappa )=\{p=(p_0,p_1,p_2)\in\LL^3\mid \langle p, p\rangle =1/\kappa , p_0>0 \}\subset \LL^3.$$ 
Here $p_0$ denotes the timelike coordinate in Lorentz space $\LL^3$ and $\langle p, p\rangle=-p_0^2+p_1^2+p_2^2$.

This construction is inspired by the one given in \cite{fmajm} for surfaces in $\h^2\times\R$. In particular, it agrees with that one when $\tau=0$.

\begin{thm}
Let $X:\Sigma\to\E$ be a nowhere vertical surface oriented so that  its unit normal vector field $N$ points upwards. Define  $X_*=\pi\circ X:\Sigma\to\h^2(\kappa )\subset\LL^3$, $N_*=\rmd\pi(N):\Sigma\to\LL^3$ and $\nu=\langle N,\xi\rangle$. 

Then the map $\tilde{g}:\Sigma\to\h^2=\h^2(-1)$ given by 
$$\tilde{g}=\frac{1}{\nu}\left( 2c X_* + N_*\right).$$ 
agrees with the Gauss map of $X$ as defined in Corollary \ref{ginframe}.
\end{thm}

\begin{proof} We first observe that $\tilde{g}$ actually takes values in $\h^2$. Indeed, since  $N_*$ takes values in $\rmT_{X_*(\cdot)}\h^2(\kappa ) = \langle X_*(\cdot)\rangle^\perp \subset\LL^3,$ we have $\langle N_*,X_*\rangle =0$. Also, since $\pi$ is a Riemannian submersion, we have $\langle N,N\rangle=\langle N_*,N_*\rangle+\langle N,\xi\rangle^2$, from where we get  $\langle N_*,N_*\rangle = 1-\nu^2$, and so $\langle\tilde{g},\tilde{g}\rangle = -1$. Finally, since $N$ points upwards, $\nu$ is positive, and therefore $\langle \tilde{g}, X_*\rangle =-1/(2c\nu)<0$, implying that the timelike coordinate of $\tilde{g}$ is  positive. 

Let $F :\mathbb{D} \to  \h^2$ be the natural isometry given by 
$$ 
F(w)= \frac{1}{1-|w|^2}\big(1+|w|^2,2\mbox{Re}(w),2\mbox{Im}(w)\big)
$$
and $\tilde{\Pi}:  \rmU\E \to  \h^2 $ 
the map so that  $\tilde{g}=\tilde{\Pi}\circ N$, that is, for $x\in\E$ and $Z\in\rmU_x\E$, $$\tilde{\Pi}(x,Z)=\tilde{\Pi}_x(Z)=\frac{1}{Z_3} (2c  x_*   + Z_*  ),$$ where $x_*=\pi(x)\in\h^2(\kappa )$, $Z_*=\rmd_x\pi(Z)$ and $Z_3=\langle Z,\xi\rangle$.

Let us prove that $\tilde{\Pi}=F\circ \Pi$, where $\Pi$ is the map given in Theorem \ref{thm:gauss}.  
Since $\tilde{\Pi}$ satisfies hypothesis (a) in abovementioned theorem, following the ideas of the proof of this result, it suffices to check that both maps agree at the origin $O=((\frac{1}{2c},0,0),0)\in\E$. 

Since 
$Z_*=\rmd_O\pi(Z)\in\rmT_{(\frac{1}{2c},0,0)}\h^2=
\langle (\frac{1}{2c},0,0)\rangle ^\perp\subset \LL^3$, 
then its timelike coordinate vanishes, and so 
$$\tilde{\Pi}_{O}(Z)=  \frac{1}{Z_3}(1,Z_1,Z_2).$$
On the other hand, since $\Pi_O(Z)=\frac{1}{1+Z_3}(Z_1+iZ_2) $, it is straightforward to check that $\tilde{\Pi}_{O}(Z)=F(\Pi_O(Z))$.
\end{proof}


\begin{rem}
It is possible then to follow the computations in \cite{fmajm} and find an alternative Weierstrass-type representation for surfaces of critical CMC in $\E$ in terms of the Gauss map, using this model. This method uses the fact that a nowhere antiholomorphic harmonic map into $\h^2$ is (locally) the Gauss map of a CMC 1/2 surface in $\LL^3$ \cite{akni}. In this setting, one first needs to  integrate a differential system to obtain $\eta=2\langle X_z,\xi\rangle$ and then $X_*$ is obtained from $\eta$ and the harmonic map without integration. This gives a representation formula equivalent to \eqref{repfor} but that differs significantly from it.
\end{rem}

\begin{rem}
The condition for being the Gauss map in $\h^2\times\R$  that we obtain in the present paper (being nowhere antiholomorphic) is different from the one obtained in \cite{fmajm} (the existence of Weierstrass data). This difference comes from the fact that in \cite{fmajm} the orientation induced by the Riemann surface $\Sigma$ is not assumed to coincide with the orientation induced by the unit normal for which the angle function of the surface is positive. In the present paper, we always assume that the orientation of $\Sigma$ is coherent with the choice of the unit normal of the immersion.
\end{rem}

\bibliographystyle{plain}
\bibliography{psl}

\end{document}